\documentclass{amsart}
\usepackage[english]{babel}
\usepackage[utf8]{inputenc}
\usepackage{amsmath, amssymb,epic,graphicx,mathrsfs,enumerate}
\usepackage[all]{xy}
\usepackage{color}
\usepackage{enumitem}
\usepackage{hyperref}
\usepackage[]{frontespizio}

\usepackage{amsmath}
\usepackage{amsthm}
\usepackage{amssymb}
\usepackage{latexsym}
\usepackage{epsfig}

\DeclareMathOperator{\soc}{soc}

\DeclareMathOperator{\alt}{Alt}

\newcommand{\Om}{\Theta}
\newcommand{\La}{\Omega}

\newtheorem{thm}{Theorem}
\newtheorem{cor}[thm]{Corollary}
 \newtheorem{lemma}[thm]{Lemma}
\newtheorem{prop}[thm]{Proposition}

\newtheorem{remark}[thm]{Remark}

\numberwithin{equation}{section}

\renewcommand{\footnote}{\endnote}
\newcommand{\ignore}[1]{}\makeglossary

\begin{document}
	\bibliographystyle{amsplain}

\title[Probabilistic properties of profinite groups]{Probabilistic properties of profinite groups}
\author{Eloisa Detomi}
\address{Eloisa Detomi\\ Universit\`a di Padova\\  Dipartimento di Matematica \lq\lq Tullio Levi-Civita\rq\rq\\ Via Trieste 63, 35121 Padova, Italy\\email: eloisa.detomi@unipd.it}
\author{Andrea Lucchini}
\address{Andrea Lucchini\\ Universit\`a di Padova\\  Dipartimento di Matematica \lq\lq Tullio Levi-Civita\rq\rq\\ Via Trieste 63, 35121 Padova, Italy\\email: lucchini@math.unipd.it}
\author{Marta Morigi}
\address{Marta Morigi\\ Universit\`a di Bologna\\  Dipartimento di Matematica \\Piazza di Porta San Donato 5, 40126 Bologna, Italy\\email: marta.morigi@unibo.it}
\author{Pavel Shumyatsky}
\address{Pavel Shumyatsky\\ University of Brasilia
Department of Mathematics\\  Brasilia-DF, 70910-900, Brazil\\e-mail: pavel@unb.br}
\thanks{The first three authors are members of GNSAGA (INDAM).  The fourth author was partially supported by FAPDF and CNPq.}
\begin{abstract} Let $\mathfrak C$ be a class of finite groups which is closed for subgroups, quotients and direct products. Given a profinite group $G$  and an element $x\in G$, we denote by $P_{\mathfrak{C}}(x,G)$ the probability that $x$ and a randomly chosen element of $G$  generate a  pro-${\mathfrak C}$ subgroup. We say that a profinite group $G$ is $\mathfrak C$-positive if $P_{\mathfrak{C}}(x,G)>0$ for all $x \in G.$ 
We establish several equivalent conditions for a profinite group to be  $\mathfrak C$-positive when  $\mathfrak C$ is the class of finite soluble groups or of finite nilpotent groups. In particular, for the above classes, the profinite  $\mathfrak C$-positive  groups are virtually prosoluble (resp., virtually nilpotent). We also draw some consequences on the prosoluble (resp. pronilpotent) graph of a profinite group.
\end{abstract}
\maketitle

\hbox{}

\section{Introduction}
Let $\mathfrak C$ be a class of finite groups which is closed for subgroups, quotients and direct products. Given a profinite group $G$  and an element $x\in G$, we are interested in the probability that a randomly chosen element of $G$  generates a  pro-${\mathfrak C}$ subgroup together with $x$. 

 We denote by
$\Omega_{\mathfrak C}(x,G)$ the subset of $G$ consisting of elements $g\in G$ with the property that $\langle x,g\rangle$ is a pro-${\mathfrak C}$ 
 group  (in Section \ref{sec:2} we will show that this set is closed). 
Let $\mu$ be the normalized 
 Haar measure on $G$. Then the probability that a randomly chosen element of $G$  generates a  pro-${\mathfrak C}$ subgroup together with $x$ is $P_{\mathfrak{C}}(x,G)=\mu(\Omega_{\mathfrak C}(x,G))$. We may also define $$\Omega_{\mathfrak C}(G)=\bigcap_{x\in G}\Omega_{\mathfrak C}(x,G)$$
and compute $\mu(\Omega_{\mathfrak C}(G)).$

 We say that a profinite group $G$ is $\mathfrak C$-positive if $P_{\mathfrak{C}}(x,G)>0$ for all $x \in G.$ Moreover we say that $G$ is $\mathfrak C$-bounded-positive if there exists a positive constant $\eta$ such that $P_{\mathfrak{C}}(x,G)>\eta$ for all $x \in G.$

Note that if $\mathfrak A$ is the class of the finite abelian groups, then 
 $\Omega_{\mathfrak A}(x,G)=C_G(x)$, $P_{\mathfrak{C}}(x,G)=|G:C_G(x)|^{-1}$, $\Omega_{\mathfrak A}(G)=Z(G)$ and $\mu(\Omega_{\mathfrak A}(G))=|G:Z(G)|^{-1}.$
In particular a profinite group $G$ is $\mathfrak A$-positive if and only if it is an FC-group and is $\mathfrak A$-bounded-positive if and only if it is a BFC-group. It follows from \cite[Lemma 2.5]{as} that the following are equivalent:
\begin{enumerate}
	\item $G$ is $\mathfrak A$-positive;
	\item $G$ is is $\mathfrak A$-bounded-positive;
	\item $\mu(\Omega_{\mathfrak A}(G)) > 0.$
\end{enumerate}
This suggests to compare the properties that $G$ is $\mathfrak C$-positive, $G$ is $\mathfrak C$-bounded-positive and $\mu(\Omega_{\mathfrak C}(G)) > 0$
for other choices of $\mathfrak C.$

Denote by $P_{\mathfrak C}(G,G)$ the probability that two randomly chosen elements of a profinite group $G$ generate a pro-$\mathfrak C$ subgroup. A crucial result in approaching our problem is the following consequence of the Baire category theorem.

\begin{thm}\label{ovvissimo}
Let $G$ be a profinite group. If $G$ is  $\mathfrak C$-positive, then $P_{\mathfrak C}(G,G)>0$.
\end{thm}

The probabilities  that two randomly chosen elements of a finite group generate a soluble (respectively, nilpotent) subgroup have  been studied by J.S. Wilson in \cite{swil} and \cite{nwil}, respectively. The main theorems therein have as a direct consequence the following results on profinite groups (see \cite{nwil}).  Denote by $P_{\mathfrak S}(G,G)$ and $P_{\mathfrak N}(G,G)$ the probabilities that two randomly chosen elements of a profinite group $G$ generate a prosoluble (respectively, pronilpotent) subgroup; then $G$  is virtually prosoluble if and only if  $P_{\mathfrak S}(G,G)>0$ and it is virtually pronilpotent if and only if  $P_{\mathfrak N}(G,G)>0$.

If $G$ is a finite group, then $\Omega_{\mathfrak S}(x,G)$ is the so called \lq solubilizer\rq \ of $x$ in $G.$
In general it is not a subgroup, however when $G$ is a finite group, it follows from \cite[Theorem 1.1]{gu} that the intersection $\Omega_{\mathfrak S}(G)$ of the solubilizers coincides with the soluble radical $R(G)$ of $G.$
This implies that if $G$ is an arbitrary profinite group, then $\Omega_{\mathfrak S}(G)$ is the prosoluble radical of $G$, i.e. the largest normal prosoluble subgroup of $G$ (see \cite[Proposition 1.9]{hl}). We will use the symbol $R(G)$ also for the prosoluble radical of a profinite group $G.$
Combining Theorem \ref{ovvissimo} with Wilson's result, we immediately obtain the following:

\begin{thm}\label{soluble}
	Let $G$ be a profinite group and let $\mathfrak S$ be the class of the finite soluble groups. The following are equivalent:
	\begin{enumerate}
		\item $G$ is $\mathfrak S$-positive;
		\item $G$ is  $\mathfrak S$-bounded-positive;
			\item $P_{\mathfrak S}(G,G)>0 $; 
			\item $G$ is virtually prosoluble;
		\item $\mu(\Omega_{\mathfrak S}(G))=|G:R(G)|^{-1}>0.$ 
	\end{enumerate}
\end{thm}

The situation with the class $\mathfrak N$ of nilpotent groups is more complicated. The condition of being  $\mathfrak N$-bounded-positive is stronger than the requirement that $P_{\mathfrak N}(G,G)>0$, as the following example shows. Given an odd prime $p,$ consider the semidirect product $G=\mathbb Z_p \rtimes \langle x\rangle,$ where $\mathbb Z_p$ is the group of the $p$-adic integers, $|x|=2$ and $z^x=-z$ for every $z\in \mathbb Z_p.$ Although $G$ is virtually pronilpotent, $\Omega_{\mathfrak N}(x,G)=\langle x \rangle$ and $P_{\mathfrak N}(x,G)=0.$   If $G$ is a finite group, then $\Omega_{\mathfrak N}(x,G)$ is the so called \lq nilpotentizer\rq \ of $x$ in $G.$
In general it is not a subgroup, but the intersection $\Omega_{\mathfrak N}(G)$ of the nilpotentizers coincides with the hypercenter $Z_\infty(G)$ of $G$ (see \cite[Proposition 2.1]{az}). More generally, when $G$ is a profinite group, $\Omega_{\mathfrak N}(G)$ is the hypercenter $Z_\infty(G)$ defined as the set of all $x \in G$ such that $xN \in Z_\infty(G/N)$ for every open normal subgroup $N$ of $G$ (see \cite{hyper}).

We will establish the following theorem.

\begin{thm}\label{nilpot}
	Let $G$ be a profinite group and let $\mathfrak N$ be the class of the finite nilpotent groups. The following are equivalent:
	\begin{enumerate}
    \item $G$ is $\mathfrak N$-positive;
		\item $G$ is  $\mathfrak N$-bounded-positive;
		\item $Z_\infty(G)$ is open in $G$ and $\mu(\Omega_{\mathfrak N}(G))=|G:Z_\infty(G)|^{-1}>0.$
	\item 	$G$ is finite-by-pronilpotent.
		\end{enumerate}
\end{thm}

Given a class $\mathfrak C$ of finite groups, to any profinite group $G$ we may associate a graph $\Gamma_{\mathfrak C}(G)$ which is defined as follows: the vertices of $\Gamma_{\mathfrak C}(G)$ are the elements of $G\setminus \Omega_{\mathfrak C}(G)$ and two distinct vertices $g_1, g_2$ are adjacent if $\langle g_1, g_2\rangle$ is a pro-${\mathfrak C}$ subgroup of $G.$

 In the particular case when $\mathfrak A$ is the class of finite abelian groups, the vertices of $\Gamma_{\mathfrak A}(G)$ are the non-central elements of $G$ and 
 two distinct vertices are adjacent if and only if they commute in $G$. The graph $\Gamma_{\mathfrak A}(G)$ is known with the name of commuting graph of $G$. Commuting graphs arise naturally in many different contexts and they have been intensively studied by various authors in recent years (see in particular   \cite{GP},  \cite{ij}, \cite{mp}, \cite{parker}, \cite{SS}). 
 
  The (pro)soluble graph  $\Gamma_{\mathfrak S}(G)$ has been studied for finite groups $G$ in \cite{almm} and \cite{BLN}, where it was proved that  $\Gamma_{\mathfrak C}(G)$ is always connected and its diameter is at most 5.  An attractive property of the prosoluble graph $\Gamma_{\mathfrak S}(G)$ of a profinite group $G$ is that $\Omega_{\mathfrak S}(G)=R(G)$ is a closed normal subgroup of $G$ and two vertices $g_1, g_2$ are adjacent in $\Gamma_{\mathfrak S}(G)$ if and only if $g_1R(G), g_2R(G)$ are adjacent in
$\Gamma_{\mathfrak S}(G/R(G)).$  In particular $\Gamma_{\mathfrak{S}}(G)$ is connected if and only if $\Gamma_{\mathfrak{S}}(G/R(G))$ is connected, and the graphs $\Gamma_{\mathfrak S}(G)$ and $\Gamma_{\mathfrak{S}}(G/R(G))$ have the same diameter (see \cite[Lemma 2.2]{BLN}). Thus, 
as a consequence of Theorem \ref{soluble}, we obtain the following.

\begin{cor}\label{graph}
Let $G$ be a profinite nonprosoluble group.  If ${P_{\mathfrak S}(g,G)} > 0$ for every $g\in G,$
then the prosoluble graph $\Gamma_{\mathcal{S}}(G)$ of $G$ is connected and its diameter is at most 5.
\end{cor}

If $G$ is a finite group, then its nilpotent graph $\Gamma_{\mathfrak{N}}(G)$
is not always connected. 
 However each connected component of  $\Gamma_{\mathfrak{N}}(G)$  has diameter at most $10$ (see \cite[Proposition 7.6]{BLN}).
If $G$ is a profinite group, then $g_1, g_2$ are adjacent in $\Gamma_{\mathfrak N}(G)$ if and only if $g_1Z_\infty(G), g_2 Z\infty(G)$ are adjacent in
$\Gamma_{\mathfrak N}(G/Z_\infty(G)).$ Thus, 
as a consequence of Theorem \ref{nilpot}, we obtain the following.

\begin{cor}\label{graphn}
	Let $G$ be a profinite nonpronilpotent group.  If ${P_{\mathfrak N}(g,G)} > 0$ for every $g\in G,$
	then the pronilpotent graph $\Gamma_{\mathfrak{N}}(G)$ of $G$ has only finitely many connected components and each of these components has diameter  at most 10.
\end{cor}

It is a difficult problem to determine whether there exists a profinite nonprosoluble group whose prosoluble graph is not connected. In the case of a finite insoluble group $G$, the connectivity of the soluble graph $\Gamma_{\mathcal{N}}(G)$ is strongly related with the following property of the solubilizers in $G$: for every $g\in G$, the solubilizer $\Omega_{\mathfrak S}(g,G)$ properly contains $\langle g \rangle$ (see \cite[Corollary 3.2]{almm}). In the profinite context an analogue of the previous statement should say that if $G$ is not prosoluble, then, for every $g \in G,$ the (closed) subgroup $ {\langle g \rangle}$ is properly contained in $\Omega_{\mathfrak S}(g,G).$ However in Section \ref{solva} we will prove that this is false. 
  Namely the following holds.
 \begin{prop}\label{prop:sol}  There exists a non-prosoluble profinite  group $G$ containing an element $g$ such that the solubilizer $\Omega_{\mathfrak S}(g,G)$ coincides with the (closed) subgroup generated by $g$    in $G.$
\end{prop} 
 Thus this is an example of a result that is true in the case of finite groups but fails  
 in profinite groups. Indeed, by \cite[Theorem 1.2]{almm} if $G$ a finite nonabelian group,
then the solubilizer $\Omega_{\mathfrak S}(g,G)$ of $g\in G$ is never abelian.

\section{Proof of Theorems \ref{ovvissimo} and \ref{soluble}}\label{sec:2} 

 In the sequel $\mathfrak C$ will be a class of finite groups which is closed for subgroups, quotients and direct products. Let $G$ be a profinite group and $\mu$ the  normalized 
 Haar measure on $G$  or on some direct product $G^k $ (see \cite[18.1]{Jarden} for an introduction to the properties of the Haar measure).
  In the first part of this section  we will prove some results on the Haar measure that are rather clear for countably based profinite groups (see e.g. \cite{sub}) but are less obvious in the general case. 
 
  Let $X,Y$ be closed subsets of $G$. 
  We define  
 \[ \Om_{\mathfrak C}(X,Y)=\{(x,y)\in X\times Y \mid \langle x, y\rangle \textrm{ is a pro-}\mathfrak C\textrm{-group}\}.\]
Let $\mathcal N$ be the family of all open normal subgroups of $G.$
Given $N \in \mathcal N,$ let  $\pi: G \rightarrow G/N$ be the natural projection of $G$ on $G/N$ and set 
\[ \Om_{\mathfrak C,N}(X,Y)=\{(x,y)\in XN\times YN \mid \langle x, y\rangle N/N \in \mathfrak C\},\]
  which is a closed subset of $G\times G$, since
 \[ \Om_{\mathfrak C,N}(X,Y) = \pi^{-1} \left( \Om_{\mathfrak C}(XN/N,YN/N) \right). \]
 In particular 
\[ \Om_{\mathfrak C}(X,Y)=\bigcap_{N\in \mathcal N}\Om_{\mathfrak C, N}(X,Y), \] 
 and $ \Om_{\mathfrak C}(X,Y)$  is  a closed subset of $X\times Y.$
 
If $\mu(X), \mu(Y)> 0$  we can define the probability $P_{\mathfrak{C}}(X,Y)$
 that two randomly  chosen elements $x\in X$ and $y\in Y$ generate a pro-${\mathfrak C}$-subgroup as the conditional probability that $(g_1,g_2)\in \Om_{\mathfrak C}(G,G)$
 given that $(g_1,g_2)\in X \times Y.$ We have
\begin{equation}\label{eq:probXY}
P_{\mathfrak C}(X,Y))=\frac{\mu(\Om_{\mathfrak C}(X,Y))}{\mu(X)\mu(Y)}.
\end{equation}
 Recall that in the introduction we defined
\[\Omega_{\mathfrak C}(x,G)=\{y\in G \mid \langle x,y\rangle \textrm{ is a pro-${\mathfrak C}$ group}\}, \]
  so that the probability that $x$ and a randomly chosen element of $G$  generate a  pro-${\mathfrak C}$ subgroup is the measure 
 \[P_{\mathfrak{C}}(x,G)=\mu(\Omega_{\mathfrak C}(x,G)).\]
  Now it is clear that $\Omega_{\mathfrak C}(x,G)$ is closed, being the projection of 
 $ \Om_{\mathfrak C}(\{x\},G) $ on the second component of $G \times G$.

\begin{lemma}\label{unotanti}
Assume that $X$ is a closed subset
 of $G$ with $\mu(X)>0.$ If $P_{\mathfrak{C}}(x,G)\geq \eta$ for every $x\in X,$ then $P_{\mathfrak{C}}(X,G)\geq \eta.$
\end{lemma}

\begin{proof}
Let $\chi: G\times G \to \mathbb R$ be the characteristic function of $\Om_{\mathfrak C}(X,G)$. 
 Then, applying Fubini's Theorem,
\begin{equation}\label{fubini}
\mu(\Om_{\mathfrak C}(X,G))=\int_{G\times G} \chi(x,y) d_\mu(x,y)=
\int_G \left(\int_G \chi(x,y) d_\mu(y) \right) d_\mu(x). 
\end{equation}
Note that if $x \in X$, then 
\begin{equation*}
 \int_G \chi(x,y) d_\mu(y) = \mu(\La_{\mathfrak C}(x,G)) = P_{\mathfrak{C}}(x,G).
\end{equation*}
So, by our assumption on the elements of $X$, 
\begin{equation*}
\int_G \chi(x,y) d_\mu(y)   \geq \begin{cases}
   \eta, & \text{if $x \in X$}.\\
    0, & \text{otherwise}. 
  \end{cases}
\end{equation*}
Hence, considering the characteristic function $\psi: G \to \mathbb R$ of $X,$
we have that 
\[ \int_G \chi(x,y) d_\mu(y) \geq  \eta \psi(x).\] 
Thus, by \eqref{fubini} we get 
\[ \mu(\Om_{\mathfrak C}(X,G)) \geq \int_G \eta \psi(x)d_\mu(x)=\eta \mu(X), \]
that gives $P_{\mathfrak{C}}(X,G)\geq \eta.$
\end{proof}

Let  $G$ be a profinite group and $x\in G.$ Given an open normal 
subgroup $N$ of $G$, let $$\La_{\mathfrak C, N}(x)=\{y\in G \mid \langle x, y\rangle N/N \in \mathfrak C \}.$$

\begin{lemma}\label{lem:inf}  Let  $G$ be a profinite group and let $\mathcal N$ be the family of all open normal subgroups of $G$. For any $x\in G$ we have
\[ P_{\mathfrak{C}}(x,G)=\inf_{N \in \mathcal N} \mu \left(\La_{\mathfrak C, N}(x) \right).\]  
\end{lemma} 
\begin{proof} 
 Let $\mathcal N$ be the family of all open normal subgroups of $G$. 
Recall that 
\[ P_{\mathfrak{C}}(x,G)=\mu \left(\La_{\mathfrak C}(x,G) \right)\] 
and note that  
$\La_{\mathfrak C}(x,G)$, being closed, is equal to the intersection of all the subgroups $\La_{\mathfrak C}(x,G)  N$, where $N$ ranges $\mathcal N$. Hence, by definition of the Haar measure, 
\[ P_{\mathfrak{C}}(x,G)=\inf_{N \in \mathcal N} \mu \left(\La_{\mathfrak C}(x,G)  N \right).\]
On the other hand, 
\begin{equation}\label{eq:inter}
\La_{\mathfrak C}(x,G)= \bigcap_{N \in \mathcal N}  \La_{\mathfrak C, N}(x),
\end{equation}
 whence 
\begin{equation}\label{eq:dis}
P_{\mathfrak{C}}(x,G) \le \inf_{N \in \mathcal N} \mu \left( \La_{\mathfrak C, N}(x) \right). 
\end{equation}
We claim that \ref{eq:dis} is actually an equality.

Note that if $M_1$ and $M_2$ are open normal subgroups of $G$, then 
\[ \La_{\mathfrak C, M_1}(x)  \cap \La_{\mathfrak C, M_1}(x)  = \La_{\mathfrak C, M_1 \cap M_2}(x). \] 
In view of \ref{eq:inter} and \cite[Lemma 0.3.1 (h)]{wilson}, we have 
\[ \La_{\mathfrak C}(x,G) N =\left(  \bigcap_{M\in \mathcal N}  \La_{\mathfrak C, M}(x) \right) N 
= \bigcap_{M \in \mathcal N}  \La_{\mathfrak C, M}(x) N.\]
 As $\La_{\mathfrak C}(x,G) N $ is open and $G$ is compact, $\La_{\mathfrak C}(x,G) N $ is the intersection of finitely many 
 $\La_{\mathfrak C, M_i}(x) N$, for $i =1, \dots ,r$. Therefore, we set $M_N= \cap_{i=1}^{r} M_i$ and observe that 
  \begin{eqnarray*}
 \La_{\mathfrak C}(x,G) N =  \bigcap_{i=1, \dots, r}  \La_{\mathfrak C, M_i}(x) N 
 =\left( \bigcap_{i=1, \dots, r}  \La_{\mathfrak C, M_i}(x) \right) N 
 = \La_{\mathfrak C, M_N}(x) N. 
  \end{eqnarray*}
 It follows that 
   \begin{eqnarray*}
 P_{\mathfrak{C}}(x,G) = \mu \left(\La_{\mathfrak C}(x,G) \right) &=& \inf_{N \in \mathcal N} \mu \left(\La_{\mathfrak C}(x,G)  N \right) \\
  &=& \inf_{N \in \mathcal N} \mu \left(  \La_{\mathfrak C, M_N}(x) N \right) \\
  &\geq&  \inf_{N \in \mathcal N} \mu \left(  \La_{\mathfrak C, M_N}(x) \right) \\
   &\geq&  \inf_{M \in \mathcal N}  \mu \left( \La_{\mathfrak C, M}(x) \right), 
   \end{eqnarray*}
which together with \ref{eq:dis} gives the desired result. 
\end{proof}

\begin{cor}\label{cor:inf} Let $G$ be a profinite group and $\mathcal N$ the family of all open normal subgroups of $G$.  For any $x\in G$  we have
 \[ P_{\mathfrak{C}}(x,G)=\inf_{N \in \mathcal N} P_{\mathfrak{C}}(xN,G/N).\] 
\end{cor} 
\begin{proof} 
 Let $N \in \mathcal N$  and recall that
 $\La_{\mathfrak C, N}(x)=\{y\in G \mid \langle x, y\rangle N/N \in \mathfrak C \}.$
 Let $\pi: G \rightarrow G/N$ be the natural projection. Then 
 \[\La_{\mathfrak C, N}(x)= \pi^{-1} \left(\{yN \in G/N  \mid \langle x, y\rangle N/N \in \mathfrak C \} \right) =
 \pi^{-1} \left(\Omega_{\mathfrak C}(xN,G/N)\right),\]
 so
 \[ \mu_G \left( \La_{\mathfrak C, N}(x)  \right) = \mu_{G/N} \left(\Omega_{\mathfrak C}(xN,G/N)\right)= P_{\mathfrak{C}}(xN,G/N)) \]
 (see e.g. \cite[Proposition 18.2.2]{Jarden}). 
Now the result follows from Lemma \ref{lem:inf}. 
Namely, 
\[ P_{\mathfrak{C}}(x,G)=\inf_{N \in \mathcal N} \mu \left(\La_{\mathfrak C, N}(x) \right)=\inf_{N \in \mathcal N} P_{\mathfrak{C}}(xN,G/N). \qedhere\]
\end{proof}
The following lemma is almost obvious. It will be useful later on.
\begin{lemma}\label{lem:infprod}
Assume $(\alpha_\lambda)_{\lambda \in \Lambda}$ and $(\beta_\lambda)_{\lambda \in \Lambda}$ are two families of positive real numbers with the property that, 
for every $\lambda_1, \lambda_2  \in \Lambda$ there exists $\mu \in \Lambda$ such that 
\[ \alpha_\mu  \le \alpha_{\lambda_1} \quad \textrm{ and } \quad  \beta_\mu \le \beta_{\lambda_2}. \]
Then 
\[ \inf_{\lambda \in \Lambda} \left( \alpha_\lambda \beta_\lambda \right) = 
\left(  \inf_{\lambda \in \Lambda}  \alpha_\lambda  \right) \left(  \inf_{\lambda \in \Lambda}  \beta_\lambda  \right).\]
\end{lemma}

\begin{proof} 
Let $\alpha= \inf_{\lambda \in \Lambda}  \alpha_\lambda$ and $\beta=\inf_{\lambda \in \Lambda}  \beta_\lambda $. 
Clearly 
\[ \inf_{\lambda \in \Lambda} \left( \alpha_\lambda \beta_\lambda \right) \ge \alpha \beta. \]
Given $\epsilon >0$, choose $\eta >0$ such that $\alpha  \eta + \beta \eta + \eta^2 < \epsilon.$ 
 Then there exists two indices  $\lambda_1, \lambda_2  \in \Lambda$ such that 
 \[ \alpha_{\lambda_1} \le \alpha  + \eta  \quad \textrm{and} \quad  \beta_{\lambda_2} \le  \beta + \eta. \] 
 By hypothesis,  there exists $\mu \in \Lambda$ such that  
 $ \alpha_\mu  \le \alpha_{\lambda_1} $ and $\beta_\mu \le \beta_{\lambda_2}$. Hence 
\[ \alpha_\mu  \le  \alpha  + \eta   \quad \textrm{ and } \quad  \beta_\mu \le  \beta + \eta, \]
and therefore 
\[ \inf_{\lambda \in \Lambda} \left( \alpha_\lambda \beta_\lambda \right) \le 
\alpha_\mu  \beta_\mu \le ( \alpha  + \eta ) ( \beta + \eta) = \alpha \beta + \alpha  \eta + \beta \eta + \eta^2 \le   \alpha \beta+ \epsilon, \]
and the lemma follows. 
\end{proof}

In a natural way, the probabilistic properties of profinite groups are determined by those of their finite images. This is formalized  in the next proposition.

\begin{prop}\label{prop:infXY}  Let $X,Y$ be closed subsets of a profinite group $G$ and let $\mathcal N$ be the family of all open normal subgroups of $G$. If $\mu(X), \mu(Y)> 0$, then 
\[ P_{\mathfrak{C}}(X,Y)=\inf_{N \in \mathcal N} P_{\mathfrak{C}}(XN/N,YN/N)).\] 
\end{prop} 
\begin{proof} 
Recall that 
\begin{equation}\label{eq:interX}
 \Om_{\mathfrak C}(X,Y)=\bigcap_{N\in \mathcal N}\Om_{\mathfrak C, N}(X,Y), 
 \end{equation}
where 
$ \Om_{\mathfrak C,N}(X,Y)=\{(x,y)\in XN\times YN \mid \langle x, y\rangle N/N \in \mathfrak C\}.$

Our first claim is that 
\begin{equation}\label{eq:eqX}
 \mu \left( \Om_{\mathfrak C}(X,Y) \right) =\inf_{N\in \mathcal N} \mu \left( \Om_{\mathfrak C, N}(X,Y) \right). 
  \end{equation}

Clearly, by \eqref{eq:interX},
\begin{equation}\label{eq:leX}
 \mu \left( \Om_{\mathfrak C}(X,Y) \right) \le \inf_{N\in \mathcal N} \mu \left( \Om_{\mathfrak C, N}(X,Y) \right). 
   \end{equation}

On the other hand, 
by \eqref{eq:interX} and \cite[Lemma 0.3.1 (h)]{wilson}, we have 
\[ \Om_{\mathfrak C}(X,Y) N^2 =\left(  \bigcap_{M \in \mathcal N}  \Om_{\mathfrak C, M}(X,Y) \right) N^2 = \bigcap_{M \in \mathcal N}  \Om_{\mathfrak C, M}(X,Y) N^2.\]
 As $\Om_{\mathfrak C}(X,Y) N^2 $ is open and $G$ is compact, $\Om_{\mathfrak C}(X,Y) N^2 $ is the intersection of finitely many 
 $\Om_{\mathfrak C, M_i}(X,Y) N^2$, for $i =1, \dots ,r$. Moreover, for $M_N= \cap_{i=1}^{r} M_i$  we have that 
\[\bigcap_{i=1, \dots, r}  \Om_{\mathfrak C, M_i}(X,Y)   = \Om_{\mathfrak C, M_N}(X,Y).\]
 Therefore, 
 \begin{eqnarray*} 
 \Om_{\mathfrak C}(X,Y) N^2 = \!\!\!
 \bigcap_{i=1, \dots, r}\!\!\!  \Om_{\mathfrak C, M_i}(X,Y) N^2 
 =\left( \bigcap_{i=1, \dots, r}\!\!  \Om_{\mathfrak C, M_i}(X,Y)\right)N^2 
 = \Om_{\mathfrak C, M_N}(X,Y) N^2. 
  \end{eqnarray*}
It follows that 
   \begin{eqnarray*}
  \mu \left(\Om_{\mathfrak C}(X,Y) \right) &=& \inf_{N \in \mathcal N} \mu \left(\Om_{\mathfrak C}(X,Y)  N^2 \right) \\
  &=& \inf_{N \in \mathcal N} \mu \left(  \Om_{\mathfrak C, M_N}(X,Y) N^2 \right) \\
  &\geq&  \inf_{N \in \mathcal N} \mu \left(  \Om_{\mathfrak C, M_N}(X,Y) \right) \\
   &\geq&  \inf_{M \in \mathcal N}  \mu \left( \Om_{\mathfrak C, M}(X,Y) \right), 
   \end{eqnarray*}
 that together with \eqref{eq:leX} gives  \eqref{eq:eqX}. 
 
 For every $N \in \mathcal N$ let $\pi: G^2 \rightarrow (G/N)^2$ be the natural projection, and let $\overline{XN}$, $\overline{YN}$ be the images of $X$ and $Y$ respectively in the quotient group $G/N$. Then 
 \[\Om_{\mathfrak C, N}(X,Y)=  
 \pi^{-1} \left(\Om_{\mathfrak C}(\overline{XN},\overline{YN})\right)\]
 and so 
 \[ \mu_{G^2} \left( \Om_{\mathfrak C, N}(X,Y)  \right) =    \mu_{(G/N)^2} \left(\Om_{\mathfrak C}(\overline{XN},\overline{YN})\right) \]
 (see e.g. \cite[Proposition 18.2.2]{Jarden}). 
Hence 
   \[ \mu_{G^2} \left( \Om_{\mathfrak C}(X,Y)  \right) =   \inf_{N \in \mathcal N} \mu_{(G/N)^2} \left(\Om_{\mathfrak C}(\overline{XN},\overline{YN})\right), \]
 that is,  
\[P_{\mathfrak{C}}(X,Y)  \mu_{G^2} \left( X \times Y \right) =
  \inf_{N \in \mathcal N}  \left( P_{\mathfrak{C}}(\overline{XN},\overline{YN}) \, \mu_{(G/N)^2} \left( \overline{XN} \times \overline{YN} \right) \right).\] 
  
  Note that $\{P_{\mathfrak{C}}(\overline{XN},\overline{YN}) \}_{N \in \mathcal N}$ and $ \left\{\mu_{(G/N)^2} \left( \overline{XN} \times \overline{YN} \right) \right\}_{N \in \mathcal N} $ satisfy the assumptions of Lemma \ref{lem:infprod}. Indeed if $N_1, N_2 \in \mathcal N$, then for $M=N_1 \cap N_2$, since $G/M$ is finite, we have that 
  \[ P_{\mathfrak{C}}( XN_1/N_1, YN_1/N_1) \ge P_{\mathfrak{C}}( XM/M, YM/M)\] 
  and clearly also $    \mu_{(G/N_2)^2} \left( XN_2/N_2 \times YN_2/N_2 \right) \ge  \mu_{(G/M)^2} \left( XM/M \times YM/M \right).$ 
  So, by Lemma \ref{lem:infprod},  
  \[\begin{aligned}   &\inf_{N \in \mathcal N}  \left( P_{\mathfrak{C}}(\overline{XN},\overline{YN}) \mu_{(G/N)^2} \left( \overline{XN} \times \overline{YN} \right) \right)
   \\ &\quad \quad =  \left(  \inf_{N \in \mathcal N} P_{\mathfrak{C}}(\overline{XN},\overline{YN}) \right) \left(  \inf_{N \in \mathcal N}   \mu_{(G/N)^2} \left( \overline{XN} \times \overline{YN} \right) \right).\end{aligned}\]
As
   \[ \mu_{G^2} \left( X \times Y \right) = \inf_{N \in \mathcal N}  \mu_{(G/N)^2}  \left( \overline{XN} \times \overline{YN} \right),\] 
   whenever $\mu_{G^2} \left( X \times Y \right)= \mu(X) \mu(Y)  \neq 0$ we get that
   \[P_{\mathfrak{C}}(X,Y) = \inf_{N \in \mathcal N}  P_{\mathfrak{C}}(\overline{XN},\overline{YN}),\]
   as claimed. 
\end{proof}

\begin{remark} If $Y$ is a closed subset of $G$ and $x\in G$, we can consider the subset 
$\Omega_{\mathfrak C}(x,Y)$ of $Y$ consisting of elements $y\in Y$ with the property that $\langle x,y\rangle$ is a pro-${\mathfrak C}$ group. 
  If  $\mu(Y)> 0$, then we can define the probability $P_{\mathfrak{C}}(x,Y)$  that a randomly chosen element of $Y$  generates a  pro-${\mathfrak C}$ group together with $x$ as the conditional probability that $g \in \Omega_{\mathfrak C}(x,G)$ 
 given that $g \in Y:$ 
 \[ P_{\mathfrak{C}}(x,Y)= \frac{\mu \left( \Omega_{\mathfrak C}(x,Y)  \right)}{\mu (Y)}. \]
  Arguing as in the proof of Proposition  \ref{prop:infXY}, it can be easily proved that 
 \[ P_{\mathfrak{C}}(x,Y) 
  = \inf_{N \in \mathcal N} P_{\mathfrak{C}}(xN,YN/N).\] 
\end{remark}

\begin{cor}\label{cor:quoziente} 
 Let $\mathfrak C$ be a class of finite groups which is closed for subgroups, quotients and direct products. Let $G$ be a profinite group, $N$  a closed normal subgroup of $G$, $x \in G$  and $X, Y$  closed subsets of $G$. The following holds: 
 \begin{enumerate}
\item $P_{\mathfrak{C}}(xN,G/N)\ge P_{\mathfrak{C}}(x,G)$
\item   If $\mu(X), \mu(Y) >0$, then  $P_{\mathfrak{C}}(XN/N,YN/N)\ge P_{\mathfrak{C}}(X,Y)$.
\end{enumerate}
In particular, if $G$ is $\mathfrak{C}$-positive, then $G/N$ is also $\mathfrak{C}$-positive.
\end{cor}

\begin{proof}
Since by Corollary \ref{cor:inf},  ${P_{\mathfrak C}(x,G)}=\inf_M {P_{\mathfrak C}(xM,G/M)},$ 
 with $M$ running over the set of the open normal subgroups of $G,$ it suffices to note that, as $\mathfrak C$ is closed for quotients, 
 in every finite (continuous) quotient of $G,$ we have
 \[P_{\mathfrak S}(xMN,G/MN)\ge P_{\mathfrak S}(xM,G/M).\]
 The second statement follows from the same argument and Proposition \ref{prop:infXY}. 
\end{proof}

\begin{lemma}\label{unionefinita} 
 Let $G$ be a profinite group and assume that $X,Y$ are closed subsets of $G$ such that $X$ is the disjoint union of $r$ closed 
subsets $X_1,\dots,X_r$ with $\mu(X_i)=\alpha>0$ for every $i=1,\dots,r$ and  $Y$ is the disjoint union of $s$ closed subsets $Y_1,\dots,Y_s$ with $\mu(Y_j)=\beta>0$ for every $i=1,\dots,s$. Then
\begin{enumerate}
\item $P_{\mathfrak C}(X,Y)=\sum_{i,j}P_{\mathfrak C}(X_i,Y_j)/rs;$
\item There exist $i\in\{1,\dots,r\}$ and $j\in\{1,\dots,s\}$
  such that ${P_{\mathfrak C}(X_i,Y_j)}\ge {P_{\mathfrak C}(X,Y)}.$
\end{enumerate}
\end{lemma}
\begin{proof}
Note that $X\times Y$ is the disjoint union of the $rs$ sets $X_i\times Y_j$, with $i=1,\dots,r$,  $j=1,\dots,s.$
Thus 
\[{P_{\mathfrak C}(X,Y)}=\frac{\mu(\Om_{\mathfrak C}(X,Y))}{\mu(X)\mu(Y)}=\frac{\sum_{i,j}\mu(\Om_{\mathfrak C}(X_i,Y_j))}{r\alpha s\beta}=\frac{\sum_{i,j}{P_{\mathfrak C}(X_i,Y_j)}}{rs}.\]
This proves (1). The other claim is a straightforward consequence of (1).
\end{proof}

Now we are ready to prove Theorem \ref{ovvissimo}. 

\begin{proof}[Proof of Theorem \ref{ovvissimo}] 
Assume that $G$ is  $\mathfrak C$-positive. 
	For $n\in \mathbb N,$ let
	$$X_n:=\left\{g\in G \mid {P_{\mathfrak C}(g,G)} \geq {1}/{n}\right\}.$$
	Let $\mathcal N$ be the family of all open normal subgroups of $G$. 
	 For any $N \in \mathcal N$, the set 
	 \[X_{n,N}=\{g\in G \mid P_{\mathfrak C}(gN,G/N) \geq 1/n\} \] 
	is a union of cosets of $N$, and in particular it is a closed subset of $G$. 
	Since, by Corollary \ref{cor:inf}, 
		 \[ P_{\mathfrak{C}}(g,G)=\inf_{N \in \mathcal N} P_{\mathfrak{C}}(gN,G/N),\] 
        it is clear that 
	\[X_n=\cap_{N\in \mathcal N}X_{n,N}.\]
	Therefore $X_n$ is a closed subset of $G.$
	  
	 Since $G$ is  $\mathfrak C$-positive, we have that 
	 \[ G= \bigcup_{n \in \mathbb N} X_n, \] 
	 so, by the Baire category theorem, there exists $n\in \mathbb N$ such that $X_n$ contains a non-empty open subset. Thus
	$G$ contains an open normal subgroup $N$ and an element $x$ such that
	${P_\mathfrak{C}}(g,G)\geq \eta>0$ for every $g\in xN$. Here $\eta$ is a suitable positive real number. It follows from Lemmas \ref{unionefinita} and \ref{unotanti} that
\[{P_\mathfrak{C}(G,G)}= \frac{1}{|G:N|}\sum_{tN \in G/N} P_\mathfrak{C}(tN,G)\ge \frac{P_\mathfrak{C}(xN,G)}{|G:N|} \ge  \frac{\eta}{|G:N|}. \qedhere\] 
	\end{proof}
Theorem \ref{soluble} now follows easily.

\begin{proof}[Proof of Theorem \ref{soluble}] 
Let $G$ be a profinite group with the prosoluble radical $R(G)$  and let $\mathfrak S$ be the class of the finite soluble groups. We want to prove that  the following conditions are equivalent:
	\begin{enumerate}
	\item $G$ is $\mathfrak S$-positive;
		\item $G$ is  $\mathfrak S$-bounded-positive;
			\item $P_{\mathfrak S}(G,G)>0 $ 
			\item $G$ is virtually prosoluble;
		\item $\mu(\Omega_{\mathfrak S}(G))=|G:R(G)|^{-1}>0.$ 
	\end{enumerate}
\noindent 	   $(4)$ is trivially equivalent to  $(5)$, as  $\Omega_{\mathfrak S}(G)$ is the prosoluble radical of $G$.

\noindent We  prove that (4) implies (2). Since $\langle g, R(G) \rangle$ is prosoluble for every $g \in G$, that is $ R(G) \leq \Omega_{\mathfrak S}(g,G)$, we have that if $G$ is  virtually prosoluble, then  $P_{\mathfrak S}(g,G)=\mu(\Omega_{\mathfrak S}(g,G))\geq \mu(R(G))=|G:R(G)|^{-1}.$ Hence $G$ is  $\mathfrak S$-bounded-positive. 

\noindent 	   $(2)$  trivially implies  $(1)$.

\noindent 	  $(1)$ implies $(3)$ by Theorem \ref{ovvissimo}.

\noindent 	  $(3)$ implies$ (4)$ by Wilson's result in  \cite{nwil}.
\end{proof}

\section{The class of finite nilpotent groups}

Here we will prove Theorem \ref{nilpot}. We start with some technical observations. The following lemma is a slight generalization of \cite[Lemma 1(a)]{nwil}.

\begin{lemma}\label{qnhall}
Let $G=NQ$ be a finite group, where $Q$ is a normal nilpotent $\pi$-subgroup and $N$ is a nilpotent subgroup generated by two elements $ u, v $. Let $R$ be  the $\pi^\prime$-Hall subgroup of $N.$
Then
$P_{\mathfrak N}(uQ,vQ)\leq |Q:C_Q(R)|^{-1},$
\end{lemma}
\begin{proof}
	Note that $G$ is soluble. Let $C=N_Q(R)$; as $[C,R]\le Q\cap R=1$ it follows that $C=C_Q(R).$ Let $a, b \in  Q$ and $H=Q\langle ua, vb \rangle$. 
	Clearly  $\langle ua, vb \rangle$
	contains a $\pi^\prime$-Hall subgroup of $G,$ which must be conjugate to $R$. So $R^c\leq \langle ua, vb \rangle$ for some $c\in Q.$ Since, for $c \in Q,$ the map $x\mapsto x^c$ takes each coset of $Q$ to itself, there are as many nilpotent subgroups $\langle ua, vb \rangle$ containing $R^c$ as those containing $R$. Therefore the number of pairs $(a,b)$ with $\langle ua, vb \rangle$ nilpotent is $|Q:C|$ times the number $k$ of pairs $(a,b)$ with $\langle ua, vb \rangle$ nilpotent and containing $R.$ Let us give an upper bound for $k$. If $R\leq \langle ua, vb \rangle$ and $\langle ua, vb \rangle$ is nilpotent, then $u, v, ua, vb$ normalize $R,$ so that $a, b \in N_Q(R)=C.$ Thus $k\le |C|^2$ and therefore $P_{\mathfrak N}(uQ,vQ)\leq |Q|^{-2}|Q:C||C|^2=|Q:C|^{-1}.$ 
\end{proof}

\begin{cor}\label{modw}
	Let $G=NQ$ be a profinite group, where $Q$ is an open normal pronilpotent $\pi$-subgroup and $N=\langle u, v \rangle.$ If $N$ is pronilpotent and $R$ is the $\pi^\prime$-Hall subgroup of $N$, then
	$P_{\mathfrak N}(uQ,vQ)\leq |Q:C_Q(R)|^{-1}$.
\end{cor}
\begin{proof}
	Let $\eta=P_{\mathfrak N}(uQ,vQ)$ and let $\mathcal M$ be the set of open normal subgroups of $G$ contained in $Q.$ By Proposition \ref{prop:infXY},  we have $P_{\mathfrak N}(uQ/M,vQ/M)\ge \eta$  for every $M\in \mathcal M$. Hence, by Lemma \ref{qnhall}, $|Q:C_M|\leq 1/\eta,$ where
	$C_M=\{x\in Q \mid [x,R]\leq M\}.$
	We conclude $|Q:C_Q(R)|\leq \inf_{M\in \mathcal M}|Q:C_M|\leq 1/\eta.$
\end{proof}

We are now ready to prove the main result for nilpotent groups.

\begin{proof}[Proof of Theorem \ref{nilpot}]
	Let $G$ be a profinite group  and let $\mathfrak N$ be the class of finite nilpotent groups. We want to prove that following conditions are equivalent:
	\begin{enumerate}
    \item $G$ is $\mathfrak N$-positive;
		\item $G$ is  $\mathfrak N$-bounded-positive;
		\item $\mu(\Omega_{\mathfrak N}(G))=|G:Z_\infty(G)|^{-1}>0.$
		\item 	$G$ is finite-by-pronilpotent.
		\end{enumerate}
		\smallskip

\noindent We  prove that $(1)$ implies $(3).$
By Theorem \ref{ovvissimo}, ${P_\mathfrak{N}}(G,G)>0$ and therefore,  by Wilson's result \cite{nwil}, 
 $G$ contains an open pronilpotent normal subgroup $F.$
In particular the set $\pi$
 of all prime divisors of $|G:F|$ is finite. By \cite[Theorem A]{hyper}, the hypercenter of $G$ coincides with the intersection of the normalizers of the Sylow subgroups of $G.$ Since a $p$-Sylow subgroup of $G$ is normal in $G$ when $p\notin \pi,$ it suffices to prove that, for every $p \in \pi,$ the normalizer of a $p$-Sylow subgroup of $G$ has finite index in $G$, that is, $G$ contains only finitely many $p$-Sylow subgroups. So fix $p\in \pi$ and let $P$ be a $p$-Sylow subgroup of $F$. Since $P$ is normal in $G,$ it is contained in every 
$p$-Sylow subgroup of $G$, so it suffices to prove that  $G/P$ contains only finitely many $p$-Sylow subgroups. Let $T/P$ be a $p$-Sylow subgroup of $G/P$; then $T/P$ is finite, since $T/P$ is isomorphic to a Sylow subgroup of $G/F.$ Notice that we can replace $G$ with $G/P$, which is still  $\mathfrak N$-positive by Corollary \ref{cor:quoziente}; thus we can assume that $P=1$, $T$ is finite and we want to prove that $G$ contains only finitely many $p$-Sylow subgroups.
 
Fix $z\in T$. For $n\in \mathbb N,$ let
$$Y_n:=\left\{g\in F \mid {P_{\mathfrak N}(zg,G)} \geq \frac{1}{n}\right\}.$$
Let $\mathcal M$ be the family of all open normal subgroups of $G$ contained in $F$.
  For any $M \in \mathcal M$, the set 
  $$Y_{n,M}=\left\{g\in F \mid P_{\mathfrak N}(zgM,G/M)\} \geq \frac 1n\right\}$$
is a union of cosets of $M$, and in particular it is a closed subset of $F$. Since, by Corollary \ref{cor:inf},
		 \[ P_{\mathfrak{N}}(zg,G)=\inf_{M \in \mathcal N} P_{\mathfrak{C}}(zgM,G/M),\] 
        it is clear that 
	\[Y_n=\cap_{M\in \mathcal M}Y_{n,M}.\]
	Therefore $Y_n$ is a closed subset of $G.$
  By the Baire category theorem, there exists 
  an integer $\tilde n$ such that $Y_{\tilde n}$ contains a non-empty open subset. Thus $F$ contains an open normal subgroup $K$ and an element $m$ such that ${P_\mathfrak{N}}(g,G)\geq \epsilon$ for every $g\in zmK$ and for a suitable positive real number $\epsilon.$ By Lemma \ref{unotanti}, $P_{\mathfrak N}(zmK,G)\geq \epsilon.$ As $F$ is a union of finitely many, say $r$, cosets of $K$, it follows from Lemma \ref{unionefinita} that ${P_\mathfrak{N}}(zF,G)\geq \eta=\epsilon/r$  and there exists $w$ in $G$ such that $P_{\mathfrak N}(zF, wF)\geq \eta$. In particular, there exist $u \in Fz$ and $v\in Fw$ such
that $\langle u,v\rangle$ is pronilpotent. Let $\omega$ be the set of the prime divisors of $|F|$ and let $R$ be the $\omega^\prime$-Hall subgroup of 
 $\langle u, v\rangle$. By Corollary \ref{modw}, $|F:C_F(R)|\leq \eta^{-1}.$
 On the other hand $R$ contains a conjugate $\tilde z$ of $z,$ so
 $$|G:C_G(z)|=|G:C_G(\tilde z)| \leq |G:F||F:C_F(\tilde z)| \leq |G:F||F:C_F(R)| \leq \frac{|G:F|}{\eta}.$$
 Therefore any element of $T$ has finitely many conjugates. Hence $G$ contains only finitely many $p$-elements, and consequently, finitely many $p$-Sylow subgroups, as required. 
\smallskip

\noindent	To prove that $(3)$ implies $(2)$ observe that for every $g\in G$ we have $Z_\infty(G) \subseteq \Omega_{\mathfrak N}(g,G),$ whence $P_{\mathfrak N}(g,G)=\mu(\Omega_{\mathfrak N}(g,G))\geq \mu(Z_\infty(G))=|G:Z_\infty(G)|^{-1}.$

\noindent 	It is trivial that  $(2)$ implies $ (1)$.
\smallskip

We now prove that $Z_\infty(G)$ is open in a profinite group $G$ if and only if $G$ is finite-by-pronilpotent. Indeed, let $\gamma_\infty(G)$ denote the intersection of the lower central series of $G$. If $G$ is finite-by-pronilpotent then there is an open normal subgroup $N\leq G$ such that $N\cap\gamma_\infty(G)=1$ and it is not difficult to check that $N\leq Z_\infty(G)$. Conversely, if $Z_\infty(G)$ is open then $\gamma_\infty(G)$ is finite by a variant of the Baer theorem (see \cite{kurda}).
\end{proof}

\section{Solubilizers in profinite groups}\label{solva}
The aim of this section is to construct a nonprosoluble profinite group $G$ containing an element $g$ such that $\Omega_{\mathfrak S}(g,G)={\langle g\rangle}.$

For each natural number $t$, we recursively define a pair $(G_t,g_t),$ where $G_t$ is a finite group and $g_t\in G_t.$ Let $\alpha=(1,2,3),$ $\beta=(1,2,3,4,5) \in \alt(5).$
We set $G_0=\alt(5)$ and $g_0=\alpha.$
 Now assume that $(G_i, g_i)$ has been defined for every $i\leq t.$ Let $n_t=|G_t|$ and consider 
\[G_{t+1}=\alt(5) \wr G_t =\alt(5)^{n_t} \rtimes  G_t \]
 where  the wreath product is with respect to the regular action of $G_t.$ Let $M_{t+1}=\soc(G_{t+1})=\alt(5)^{n_t}$ be the socle of $G_{t+1}.$ An element $m \in M_{t+1}$ is a sequence $(y_x)_{x\in \alt(5)}.$  Let $T$ be a left transversal of $\langle g_t \rangle$ in $G_t$, with $1\in T.$ We define $m_{t+1}=(y_x)_{x\in  \alt(5)} \in M_{t+1}$ as follows:
$$y_x=\begin{cases}
	1 & \text{ if $x\notin T,$}\\
	\alpha & \text{ if $x=1,$}\\
	\beta & \text{ if $x\in T\setminus \{1\}$}.
\end{cases}$$ 
 Then we set 
 \[ g_{t+1}= m_{t+1} g_t. \]

\begin{lemma}\label{ht} Let $G_{t+1},g_{t+1}$ be defined as above and $M_{t+1}=\soc(G_{t+1}).$ Then $\Omega_{\mathfrak S}(g_{t+1},G_{t+1})\subseteq M_{t+1}\langle g_t \rangle.$
\end{lemma}

\begin{proof} Let  $h_{t+1}=g_{t+1}^{|g_t|}\in M_{t+1}.$ We will prove that $\Omega_{\mathfrak S}(h_{t+1},G_{t+1})\subseteq M_{t+1}\langle g_t \rangle,$ then the result follows from the fact that $\Omega_{\mathfrak S}(g_{t+1},G_{t+1})
\subseteq \Omega_{\mathfrak S}(h_{t+1},G_{t+1}).$

For every $x\in G_t$ consider the projection $\pi_x: M_{t+1}\cong \alt(5)^{n_t} \to \alt(5).$ 
Set $\gamma_t=|g_t|$ and note that 
 \[ h_{t+1}=(m_{t+1} g_{t})^{\gamma_t}=m_{t+1} (m_{t+1})^{g_t^{-1}}(m_{t+1})^{g_t^{-2}}\cdots (m_{t+1})^{g_t^{-\gamma_t+1}}\in M_{t+1}. \] 
 Moreover, 
if $x \in \langle g_t \rangle$, then $x g_t^{i}=1$ for one (and only one) index $i \in \{0, \dots,  \gamma_t -1\}$, while if $x \notin \langle g_t \rangle$, 
then $x g_t^{i} \in T$ for one (and only one) index $i \in \{0, \dots,  \gamma_t -1\}$.
Therefore 
$$(h_{t+1})^{\pi_x}=\begin{cases}
\alpha&\text { if $x \in \langle g_t \rangle$},\\
\beta&\text { otherwise}.
	\end{cases}
$$
Now assume that $\rho=mz \in \Omega_{\mathfrak S}(h_{t+1}),$ with $m=(u_x)_{x\in G_t}\in M_{t+1}$ and $z\in G_t.$ Assume, by contradiction, that $z\notin \langle g_t\rangle$. If $x \in \langle g_t \rangle,$ then $xz^{-1}\notin \langle g_t \rangle$ and consequently
$$\left((h_{t+1})^\rho \right)^{ \pi_x}= \left((h_{t+1})^{\pi_{xz^{-1}}} \right)^{u_{xz^{-1}}}=\beta^{u_{xz^{-1}}}.$$

In particular $$\langle h_{t+1}, (h_{t+1})^\rho\rangle^{\pi_x}=\langle (h_{t+1})^{\pi_x}, (h_{t+1})^{\rho \pi_x}\rangle=\langle \alpha, \beta^{u_{xz^{-1}}}\rangle=\alt(5),
$$
since no proper subgroup of $\alt(5)$ can contain an element of order $3$ and and element of order $5$. It follows that $\langle h_{t+1}, (h_{t+1})^\rho\rangle$ is not soluble. Hence  $\langle h_{t+1}, \rho\rangle$ is not soluble, in contradiction with $\rho \in \Omega_{\mathfrak S}(h_{t+1}).$
\end{proof}

\begin{proof}[Proof of Proposition \ref{prop:sol}] Let $G_{t+1},g_{t+1}$ and $M_{t+1}$ be defined as above.
For every $t\in \mathbb N,$ we have an epimorphism $\phi_{t+1}: G_{t+1}\to G_t$. Let $G$ be the inverse limit of the inverse system $(G_t,\phi_t)_{t\in \mathbb N}.$ Recall that $G$ is a profinite group which can be identified with the subgroup of the cartesian product $\prod_{t\in \mathbb N}G_t$ consisting of the elements $(z_t)_{t\in \mathbb N}$ with $z_t=z_{t+1}^{\phi_{t+1}}$ for every $t \in \mathbb N.$ 
 We set $g=(g_t)_{t\in \mathbb N}$ and note that, under this identification,  $g$  is an element of $G.$ 

There is a descending chain $(N_t)_{t\in \mathbb N}$ of open normal subgroups of $G$ such that $G/N_t \cong G_t$ for every $t \in \mathbb N$ and $\cap_{t\in \mathbb N}N_t=1.$ Moreover $N_t/N_{t+1}\cong M_{t+1}\cong 	\soc(G_{t+1}).$
Suppose $x \in \Omega_{\mathfrak S}(g,G).$ Taking into account Lemma \ref{ht} for every $t\in \mathbb N$ write
$$\begin{aligned}xN_{t+1}& \in \Omega_{\mathfrak S}(gN_{t+1},G/N_{t+1})=\Omega_{\mathfrak S}(g_{t+1}N_{t+1},G/N_{t+1})\\&\leq \soc(G/N_{t+1})\langle g_tN_{t+1}\rangle = N_t\langle g\rangle/N_{t+1}.	\end{aligned}$$

We conclude that 
$x \in \cap_{t\in \mathbb N}N_t\langle g\rangle={\langle g\rangle}.$
This proves that $\Omega_{\mathfrak S}(g,G)$ is contained in, and hence equal to, the (closed) subgroup  generated by $g$. 
\end{proof}

\end{document}